\documentclass[11pt]{article}

\usepackage{epsfig,epsf,fancybox}
\usepackage{amsmath}
\usepackage{mathrsfs}
\usepackage{amssymb}
\usepackage{graphicx}
\usepackage{color}
\usepackage{algorithmic}
\usepackage{algorithm}

\newtheorem{theorem}{Theorem}[section]

\newtheorem{lemma}[theorem]{Lemma}

\,
\,

\newcommand{\be}{\begin{equation}}
\newcommand{\ee}{\end{equation}}
\newcommand{\ba}{\begin{array}}
\newcommand{\ea}{\end{array}}
\newcommand{\bpm}{\begin{pmatrix}}
\newcommand{\epm}{\end{pmatrix}}

\newcommand{\liuhao}{\fontsize{7.0pt}{\baselineskip}\selectfont}

\definecolor{Gray}{rgb}{0.5,0.5,0.5}

\begin{document}
  \title{An Efficient Inexact Newton-CG Algorithm for the Smallest Enclosing Ball Problem of Large Dimensions
  \author{Ya-Feng Liu$^\dagger$, Rui Diao\thanks{Y.-F. Liu (corresponding author) and R. Diao are with the State Key Laboratory of Scientific and Engineering Computing, Institute of Computational Mathematics and Scientific/Engineering Computing, Academy of Mathematics and Systems Science, Chinese Academy of Sciences, Beijing 100190, China (email: yafliu@lsec.cc.ac.cn, diaorui@lsec.cc.ac.cn)},~~Feng Ye$^\ddagger$, and Hongwei Liu\thanks{F. Ye and H. Liu are with the Department of Applied Mathematics, Xidian University, Xi'an 710071, China (email: fye@xidian.edu.cn, xdliuhongwei@hotmail.com)}}
  \!\!\!\!\thanks{This work was partially supported by the National Natural Science Foundation, Grants 11331012 and 11301516.}
 }

  \maketitle

\begin{abstract}
     \boldmath
In this paper, we consider the problem of computing the smallest enclosing ball (SEB) of a set of $m$ balls in $\mathbb{R}^n,$ where the product $mn$ is large. We first approximate the non-differentiable SEB problem by its log-exponential aggregation function and then propose a computationally efficient inexact Newton-CG algorithm for the smoothing approximation problem by exploiting its special (approximate) sparsity structure. The key difference between the proposed inexact Newton-CG algorithm and the classical Newton-CG algorithm is that the gradient and the Hessian-vector product are inexactly computed in the proposed algorithm, which makes it capable of solving the large-scale SEB problem. We give an adaptive criterion of inexactly computing the gradient/Hessian and establish global convergence of the proposed algorithm. We illustrate the efficiency of the proposed algorithm by using the classical Newton-CG algorithm as well as the algorithm from [Zhou. {et al.} in Comput. Opt. \& Appl. 30, 147--160 (2005)] as benchmarks.

 %





 \end{abstract}

\section{Introduction}\label{intro}
{The smallest enclosing ball (SEB) problem is considered in this paper. The SEB
problem is to find a ball with the smallest radius that can enclose
the union of all given balls $B_i~(i=1,2,\cdots,m)$ with center
$c_i$ and radius $r_i\geq0$ in $\mathbb{R}^n$, i.e., $$B_i=\left\{x\in
\mathbb{R}^n\,|\,\left\|x-c_i\right\|\leq r_i\right\}.$$ Define $$f(x)=\max_{1\leq
i\leq m}\left\{f_i(x)\right\},$$ where
\begin{equation*}
  f_i(x)=\|x-c_i\|+r_i,\,i=1,2,\ldots,m.
\end{equation*}
Then, the SEB problem can be
formulated as the following nonsmooth convex optimization problem \cite{16}:
\begin{equation}\label{nonsmooth}
\displaystyle \min_{x\in \mathbb{R}^n} f(x).
\end{equation} It is shown in \cite{16} that problem (\ref{nonsmooth}) has a unique solution.

The SEB problem arises in a large number of important applications,
often requiring that it should be solved in large dimensions,
such as location analysis \cite{location}, gap tolerant classifiers in machine learning
\cite{1,machinelearning,machinelearning2}, tuning support vector machine parameters
\cite{2}, support vector clustering \cite{4,support}, $k$-center clustering
\cite{6}, 
testing of radius clustering \cite{7}, pattern recognition \cite{1,pattern},~
and it is itself of interest as a
problem in computational geometry \cite{9,10,circle,11,12,13}.

Many algorithms have been proposed for the special case of problem (\ref{nonsmooth}) with all $r_i$ degenerating into
zero, i.e., the problem of the smallest enclosing ball of points. To
the best of our knowledge, if the points lie in low $n$-dimensional
space, methods \cite{14,15,8,fischer,points} from computational geometry community can
yield quite satisfactory solutions in both theory and practice.
Nevertheless, these approaches cannot handle most of very recent
applications in connection with machine learning \cite{machinelearning,machinelearning2} and support vector machines \cite{4,support}
that require the problem of higher
dimensions to be solved.
%

Obviously, the non-differentiable convex SEB problem \eqref{nonsmooth} can be solved directly by the subgradient method \cite{subgradient,subgradient2}. With an appropriate step size rule, the subgradient method is globally convergent. However, the subgradient method suffers a quite slow
convergence rate, and it is very sensitive to the choice of the initial step size. By introducing additional slack variables $r\in\mathbb{R},\left\{s_i\in\mathbb{R}^n\right\}_{i=1}^m,\left\{t_i\in\mathbb{R}\right\}_{i=1}^m,$ the SEB problem \eqref{nonsmooth} can be equivalently reformulated as a second order cone program (SOCP) as follows:
\begin{equation*}
\begin{array}{cl}
\displaystyle\min_{\left\{x,\,r,\,\left\{s_i\right\}_{i=1}^m,\,\left\{t_i\right\}_{i=1}^m\right\}} & r \\
\mbox{s.t.} & r-t_i=r_i,\\
& x-s_i=c_i,\\
&\|s_i\|\leq t_i.
\end{array}\end{equation*} As shown in \cite{16}, while the above SOCP reformulation of the SEB problem can be efficiently solved by using the standard software package like SDPT3 \cite{SDPT3} with special structures taking into account, it typically requires too much memory space to store intermediate data, which makes the approach prohibitively being used for solving the SEB problem with large dimensions.

Recently, various smooth approximation-based methods \cite{16,21,20,smooth,yubo,polak,polak2,smoothing,jorsc1,jorsc2} have been proposed for solving the SEB problem in high dimensions. For instance, the log-exponential aggregation function \cite{17,aggregate} was used in \cite{16} to smooth the maximum function, and then the limited-memory BFGS algorithm \cite{lbfgs} was presented to solve the resulting smoothing problem. In \cite{20}, the authors used the Chen-Harker-Kanzow-Smale (CHKS) function \cite{chks1,chks2,chks3} to approximate the maximum function, and again, applied the limited-memory BFGS algorithm to solve the smoothing approximation problem. 

{The goal of this paper is to develop a computationally efficient algorithm that
could be used to solve the SEB problems with large $mn$. Different from the existing literatures \cite{16,20,17,aggregate,chks1,chks2,chks3}, our emphasis is not to develop new smoothing techniques but to design efficient algorithms for solving the existing smoothing approximation problems by using their special structures.

The main contribution of this paper is as follows. We propose a computationally efficient inexact Newton-CG algorithm that can efficiently solve the SEB problems with large $mn.$ At each iteration, the proposed algorithm first applies the CG method to approximately solve the inexact Newton equation and obtain the search direction; and then a line search is performed along the obtained direction. The distinctive advantage of the proposed inexact Newton-CG algorithm over the classical Newton-CG algorithm is that the gradient and the Hessian-vector product are inexactly computed, which makes it more suitable to be used to solve the SEB problem of large dimensions. Under an appropriate choice of parameters, we establish global convergence of the proposed inexact Newton-CG algorithm. Numerical simulations show that the proposed algorithm takes substantially less CPU time to solve the SEB problems than the classical Newton-CG algorithm and the state-of-the-art algorithm in \cite{16}.

The rest of the paper is organized as follows. In section \ref{log-exp}, we briefly review the log-exponential aggregation function. In section \ref{our algorithm},
by taking into the special structure of the log-exponential aggregation function into consideration, the inexact Newton-CG algorithm is proposed for solving the SEB problem. Global convergence of the proposed algorithm is established in Section \ref{convergence}. Numerical results are reported in Section \ref{experiment} to illustrate the efficiency of the proposed algorithm}, and conclusion is
drawn in section \ref{conclusion}. %

\section{Review of Log-Exponential Aggregation Function}
\label{log-exp} For any $\mu>0,$ the smooth log-exponential
aggregation function of $f(x)$ in \eqref{nonsmooth} is defined
as\begin{equation}\label{smoothfunction}
  f(x;\mu)=\mu\ln\left(\sum_{i=1}^m\exp\left(f_i(x;\mu)/\mu\right)\right),
\end{equation}
where
\begin{equation*}
f_i(x;\mu)=\sqrt{\|x-c_i\|^2+\mu^2}+r_i,\,i=1,2,\ldots,m.
\end{equation*}
\begin{lemma}[\cite{16,21,17,aggregate}]\label{yinli}
The function $f(x;\mu)$ in \eqref{smoothfunction} has the following
properties:
\begin{enumerate}
  \item [(i)] For any $x\in \mathbb{R}^n$ and $0<\mu_1<\mu_2,$ we have
  $f(x;\mu_1)<f(x;\mu_2);$
  \item [(ii)] For any $x\in \mathbb{R}^n$ and $\mu>0,$ we have
  $f(x)<f(x;\mu)\leq f(x)+\mu\left(1+\ln{m}\right);$ 
  \item [(iii)] For any $\mu>0,$ $f(x;\mu)$ is
  continuously differentiable and strictly convex in $x\in \mathbb{R}^n$, and its gradient
  and Hessian are given as follows:
  \begin{align}
\nabla f(x;\mu)=&\sum_{i=1}^m \lambda_i(x;\mu) \nabla f_i(x;\mu),
\label{gradient}\\
\displaystyle \nabla^2
f(x;\mu)=&\displaystyle\sum_{i=1}^m\left(\lambda_i(x;\mu)\nabla^2f_i(x;\mu)+\frac{1}{\mu}\lambda_i(x;\mu)\nabla
f_i(x;\mu)\nabla f_i(x;\mu)^T\right)\label{hessian}\\
&\displaystyle-\frac{1}{\mu}\nabla f(x;\mu)\nabla
f(x;\mu)^T,\nonumber
  \end{align}
where \begin{align}
\lambda_i(x;\mu)&=\frac{\exp\left(f_i(x;\mu)/\mu\right)}{\displaystyle\sum_{j=1}^m\exp\left(f_j(x;\mu)/\mu\right)}\in
(0,\,1),\,i=1,2,\ldots,m,\label{lambda}\\
\nabla f_i(x;\mu)&=\dfrac{x-c_i}{{g_i(x;\mu)}},\,i=1,2,\ldots,m,
\label{ggi}\\
\nabla^2f_i(x;\mu)&=\displaystyle\frac{I_n}{{g_i(x;\mu)}}-\dfrac{(x-c_i)(x-c_i)^T}{g_i(x;\mu)^{3}},\,i=1,2,\ldots,m,\label{hi}\\
g_i(x;\mu)&=\sqrt{\|x-c_i\|^2+\mu^2},\,i=1,2,\ldots,m,\label{gi}
\end{align}
and $I_n$ denotes the $n\times n$ identity matrix.
\end{enumerate}
\end{lemma}
%
It can be easily seen from \eqref{lambda}, \eqref{ggi}, and \eqref{gi} that, for any $\mu>0,$
\begin{align}
  \sum_{i=1}^m\lambda_{i}(x;\mu)=&~1,\label{1sum}\\
  \left\|\nabla f_i(x;\mu)\right\|<&~1,~i=1,2,\ldots,m.\label{1gradient}
\end{align}Combining \eqref{gradient}, \eqref{1sum}, and \eqref{1gradient}, we further obtain
\begin{equation}\label{1gra}
  \left\|\nabla
    f(x;\mu)\right\|<1.
\end{equation}

The algorithm proposed in \cite{16} is based on the log-exponential aggregation function \eqref{smoothfunction}. We rewrite it as Algorithm \ref{alg1} in this paper as follows.
\begin{algorithm}\caption{}\label{alg1}
\begin{algorithmic}[1]
\STATE Let $\sigma\in(0,1),~\epsilon_1,~\epsilon_2\geq 0,~x_0\in
\mathbb{R}^n,$ and $\mu_0>0$ be given, and set $k=0.$
\REPEAT
\STATE  Use the limited-memory BFGS algorithm \cite{lbfgs} to solve problem 
  \begin{equation}\label{sub}
  \min_{x\in\mathbb{R}^n} f(x;\mu_k),
  \end{equation}
 and obtain $x_k$ such that $\left\|\nabla f(x_k;\mu_k)\right\|\leq \epsilon_2.$
\STATE Set $\mu_{k+1}=\sigma\mu_k$ and $k=k+1.$
\UNTIL {$\mu_k\leq\epsilon_1$}
  \end{algorithmic}
\end{algorithm}

\begin{theorem}(\label{convergenceZhou}\cite{16}) Let $\epsilon_1=\epsilon_2=0$ in Algorithm \ref{alg1}.
Suppose that $\{x_k\}_{k\geq1}$ be the sequence generated by
Algorithm \ref{alg1} and $x^*$ be the unique solution to problem
\eqref{nonsmooth}. Then $$\lim_{k\rightarrow+\infty}x_k=x^*.$$
\end{theorem}

In the next section, we shall exploit the special (approximate) sparsity property of the log-exponential aggregation function $f(x;\mu)$ and propose an inexact Newton-CG algorithm for solving the smoothing approximation problem \eqref{sub} and thus the SEB problem \eqref{nonsmooth}. 


%

\section{Inexact Newton-CG Algorithm} \label{our algorithm}
As can be seen from \eqref{gradient} and \eqref{hessian}, the gradient $\nabla f(x;\mu)$ and Hessian $\nabla^2 f(x;\mu)$ of $f(x;\mu)$ are (convex) combinations of $\nabla f_i(x;\mu)$ and $\nabla^2 f_i(x;\mu)~(i=1,2,\ldots,m)$ with the vector $$\lambda(x;\mu)=\left(\lambda_1(x;\mu),\lambda_2(x;\mu),\ldots,\lambda_m(x;\mu)\right)^T$$ being the combination coefficients. As the parameter $\mu$ gets smaller, a large number of $\lambda_i(x;\mu)~(i=1,2,\ldots,m)$ become close to zero and thus are neglectable. To see this clearly, we define
\begin{equation*}\label{fxmu}
f_{\infty}(x;\mu)=\max_{1\leq i\leq m} \left\{f_i(x;\mu)\right\}.
\end{equation*} 
Since
$$1<\displaystyle\sum_{i=1}^m\exp\left(\left(f_i(x;\mu)-f_{\infty}(x;\mu)\right)/\mu\right)\leq m$$ and
\begin{align*}
    \lambda_i(x;\mu)=\frac{\displaystyle\exp\left(f_i(x;\mu)/\mu\right)}{\displaystyle\sum_{j=1}^m\exp\left(f_j(x;\mu)/\mu\right)}=\frac{\displaystyle\exp\left(\left(f_i(x;\mu)-f_{\infty}(x;\mu)\right)/\mu\right)}{\displaystyle\sum_{j=1}^m\exp\left(\left(f_j(x;\mu)-f_{\infty}(x;\mu)\right)/\mu\right)},
\end{align*}it follows that \begin{equation}\label{sparse}
  \displaystyle\frac{\exp\left(\left({f_i(x;\mu)-f_{\infty}(x;\mu)}\right)/{\mu}\right)}{m}\leq\lambda_i(x;\mu)<\displaystyle\exp\left(\left({f_i(x;\mu)-f_{\infty}(x;\mu)}\right)/{\mu}\right),~i=1,2,\ldots,m.
\end{equation} The second inequality of \eqref{sparse} shows that if $\mu$ is sufficiently small or $f_i(x;\mu)$ is much smaller than $f_{\infty}(x;\mu)$, then $\lambda_i(x;\mu)$ is approximately equal to zero, and $f_i(x;\mu)$ has little contribution to $f(x;\mu)$ in \eqref{smoothfunction}.
%
%


Motivated by the above observation of the (approximate) sparsity of the vector $\lambda(x;\mu)$, we propose to compute $\nabla f(x;\mu)$ and $\nabla^2 f(x;\mu)$ in an inexact way by judiciously neglecting some terms associated with very small $\lambda_i(x;\mu).$ In such a way, the computational cost is significantly reduced (compared to compute $\nabla f(x;\mu)$ and $\nabla^2 f(x;\mu)$ exactly). Then, we propose an inexact Newton-CG algorithm to solve the smoothing approximation problem \eqref{sub}. The search direction in the inexact Newton-CG algorithm is computed by applying the CG method to solve the inexact Newton equation in an inexact fashion. 
%

\subsection{An adaptive criterion and error analysis}\label{sectruncated}
%
In this subsection, we give an adaptive criterion of inexactly computing the gradient/Hessian and analyze the errors between the inexact gradient/Hessian and the true ones. For any given $\epsilon_3\in(0,1],\;\mu\in(0,1],$ define \begin{equation}\label{set}
{S(x;\mu,\epsilon)}=\left\{\,i\,|\,\lambda_i(x;\mu)\geq\epsilon\right\}
\end{equation}
with \begin{equation}\label{epsilon}
  \epsilon=\frac{\mu\epsilon_3}{10m}.
\end{equation} It is simple to see $S(x;\mu,\epsilon)\neq \emptyset.$ Otherwise, suppose $S(x;\mu,\epsilon)=\emptyset.$ Then it follows from \eqref{epsilon} and the facts $\epsilon_3\leq 1$ and $\mu\leq 1$ that
$$\sum_{i=1}^m\lambda_i(x;\mu)<\sum_{i=1}^m\epsilon=\frac{\mu\epsilon_3}{10}<1,$$ which contradicts \eqref{1sum}. 
%
%
Hence, it makes sense to define \begin{align}
\tilde f(x;\mu)=&~\mu\ln\left(\sum_{i\in
{S(x;\mu,\epsilon)}}\exp\left(
f_i(x;\mu)/\mu\right)\right),\label{tildef} \\
\nabla\tilde f(x;\mu)=&\displaystyle\sum_{i\in
{S(x;\mu,\epsilon)}}\tilde\lambda_i(x;\mu)\nabla
f_i(x;\mu),\label{tildeg}\\
 \displaystyle
\nabla^2 \tilde f(x;\mu)=&\displaystyle\sum_{i\in
S(x;\mu,\epsilon)}\left(\tilde\lambda_i(x;\mu)\nabla^2f_i(x;\mu)+\frac{1}{\mu}\tilde\lambda_i(x;\mu)\nabla
f_i(x;\mu)\nabla f_i(x;\mu)^T\right)\label{tildehessian}\\
&\displaystyle-\frac{1}{\mu}\nabla\tilde f(x;\mu)\nabla\tilde
f(x;\mu)^T,\nonumber
\end{align}
where \begin{equation*}\label{tildel} \tilde
\lambda_i(x;\mu)=\frac{\exp\left(f_i(x;\mu)/\mu\right)}{\displaystyle\sum_{j\in
{S(x;\mu,\epsilon)}}\exp\left(f_j(x;\mu)/\mu\right)}\in
(0,\,1),\,i\in {S(x;\mu,\epsilon)}.
\end{equation*}

According to (ii) of Lemma \ref{yinli}, we have\begin{equation}\label{pointappro} f(x)<\tilde f(x;\mu)\leq
f(x;\mu)\leq f(x)+\mu\left(1+\ln{m}\right),
\end{equation}where the second inequality holds with ``=" if and only if
${S(x;\mu,\epsilon)}$ defined in \eqref{set} equals the set
$\left\{1,2,\ldots,m\right\}.$ Inequality \eqref{pointappro} gives a nice explanation of $\tilde f(x;\mu)$ defined in \eqref{tildef}. For any given $\mu\in(0,1],$ (ii) of Lemma \ref{yinli} shows $f(x;\mu)$ is a uniform approximation to $f(x)$, while $\tilde f(x;\mu)$ could be explained as a ``better'' point-wise approximation to $f(x)$ (compared to $f(x;\mu)$).
%

%
%
%

The error estimations associated with \eqref{set} is given in the following theorem.
\begin{theorem}\label{errorth}Given $\epsilon_3\in(0,1],\;\mu\in(0,1],$ let $\epsilon,$ $S(x;\mu,\epsilon),$ $\tilde f(x;\mu),$ $\nabla \tilde
f(x;\mu),$ and $\nabla^2\tilde
f(x;\mu)$ be defined as in \eqref{epsilon}, \eqref{set},~\eqref{tildef},~\eqref{tildeg},~and~\eqref{tildehessian}, respectively.
Then, there hold \begin{align}
f(x;\mu)-\tilde f(x;\mu)\leq&~\mu^2\epsilon_3/9,\label{ferror}\\
\left\|\nabla f(x;\mu)-\nabla \tilde
f(x;\mu)\right\|\leq&~\mu\epsilon_3/5,\label{gerror}\\
\left\|\nabla^2f(x;\mu)-\nabla^2\tilde
f(x;\mu)\right\|\leq&~4\epsilon_3/5.\label{herror}
\end{align}
\end{theorem}
\begin{proof}
We first prove \eqref{ferror} holds true. It follows from \eqref{1sum} and \eqref{set} that
\begin{equation}\label{lambdabound}\displaystyle\sum_{i\notin {S(x;\mu,\epsilon)}}\lambda_i(x;\mu)\leq m\epsilon,~
\displaystyle\sum_{i\in {S(x;\mu,\epsilon)}}\lambda_i(x;\mu)\geq
1-m\epsilon.\end{equation}
Recalling the definitions of $f(x;\mu)$ and $\tilde
    f(x;\mu)$ (cf. \eqref{smoothfunction} and \eqref{tildef}), we obtain
\begin{align*}
    f(x;\mu)-\tilde
    f(x;\mu)~&=~\mu\ln\left(\displaystyle\sum_{i=1}^m\exp\left(f_i(x;\mu)/\mu\right)\right)-\mu\ln\left(\displaystyle\sum_{i\in {S(x;\mu,\epsilon)}}\exp\left(
f_i(x;\mu)/\mu\right)\right)\\
&\overset{(a)}{\leq}~\mu\dfrac{\displaystyle\sum_{i\notin
{S(x;\mu,\epsilon)}}\exp\left(
f_i(x;\mu)/\mu\right)}{\displaystyle\sum_{i\in
{S(x;\mu,\epsilon)}}\exp\left(
f_i(x;\mu)/\mu\right)}=~\mu\dfrac{\displaystyle\sum_{i\notin
{S(x;\mu,\epsilon)}}\lambda_i(x;\mu)}{\displaystyle\sum_{i\in
{S(x;\mu,\epsilon)}}\lambda_i(x;\mu)}\\
&~\leq~\mu\dfrac{m\epsilon}{1-m\epsilon}\leq\frac{\mu^2\epsilon_3}{9},~(\text{from}~\eqref{lambdabound}~\text{and}~\text{\eqref{epsilon}})
\end{align*}where $(a)$ comes from the fact that $\ln(1+x)\leq x$ for any $x\geq 0.$
%
%

Now we prove \eqref{gerror} holds true. Since
\begin{equation}\label{lambdab2}
    \displaystyle\sum_{i\in {S(x;\mu,\epsilon)}}\tilde\lambda_i(x;\mu)-\sum_{i\in
    {S(x;\mu,\epsilon)}}\lambda_i(x;\mu)
    =1-\displaystyle\sum_{i\in
    {S(x;\mu,\epsilon)}}\lambda_i(x;\mu)
    =\displaystyle\sum_{i\notin
    {S(x;\mu,\epsilon)}}\lambda_i(x;\mu)\geq0,
\end{equation}
and $$\tilde \lambda_i(x;\mu)>\lambda_i(x;\mu),~\forall~i\in S(x;\mu,\epsilon),$$
it follows from \eqref{gradient} and \eqref{tildeg} that
\begin{align}\label{in2}
  &~~~\left\|\nabla f(x;\mu)-\nabla\tilde f(x;\mu)\right\|\\\nonumber
  &=\left\|\displaystyle\sum_{i\in {S(x;\mu,\epsilon)}}\left(\lambda_i(x;\mu)-\tilde\lambda_i(x;\mu)\right)\nabla f_i(x;\mu)+\displaystyle\sum_{i\notin {S(x;\mu,\epsilon)}}\lambda_i(x;\mu)\nabla
  f_i(x;\mu)\right\|\\\nonumber
  &\leq\displaystyle\sum_{i\in
  {S(x;\mu,\epsilon)}}\left(\tilde\lambda_i(x;\mu)-\lambda_i(x;\mu)\right)\left\|\nabla
  f_i(x;\mu)\right\|+\displaystyle\sum_{i\notin
  {S(x;\mu,\epsilon)}}\lambda_i(x;\mu)\left\|\nabla
  f_i(x;\mu)\right\|\\\nonumber
  &\leq2\displaystyle\sum_{i\notin
    {S(x;\mu,\epsilon)}}\lambda_i(x;\mu)~(\text{from}~\eqref{1gradient}~\text{and}~\eqref{lambdab2})\\\nonumber
    &\leq 2m\epsilon=\mu\epsilon_3/5.~(\text{from}~\eqref{lambdabound}~\text{and}~\text{\eqref{epsilon}})
\end{align}

Finally, we show \eqref{herror} is also true. Combining \eqref{hessian} and \eqref{tildehessian} yields
\begin{equation}\label{term}
  \begin{array}{rl}
&~~~\left\|\nabla^2f(x;\mu)-\nabla^2\tilde
f(x;\mu)\right\|\\[8pt]
&\leq \underbrace{\left\|\displaystyle\sum_{i=1}^m\lambda_i(x;\mu)\nabla^2f_i(x;\mu)-\displaystyle\sum_{i\in
{S(x;\mu,\epsilon)}}\tilde\lambda_i(x;\mu)\nabla^2f_i(x;\mu)\right\|}_{\text{Term
A}}\\
&+\displaystyle\frac{1}{\mu}\underbrace{\left\|\displaystyle\sum_{i=1}^m\lambda_i(x;\mu)\nabla
f_i(x;\mu)\nabla f_i(x;\mu)^T-\displaystyle\sum_{i\in
{S(x;\mu,\epsilon)}}\tilde\lambda_i(x;\mu)\nabla f_i(x;\mu)\nabla
f_i(x;\mu)^T\right\|}_{\text{Term B}}\\[15pt]
&+\displaystyle\frac{1}{\mu}\underbrace{\left\|\nabla
f(x;\mu)\nabla f(x;\mu)^T-\nabla\tilde
f(x;\mu)\nabla\tilde f(x;\mu)^T \right\|}_{\text{Term C}}.
  \end{array}
\end{equation}
Noticing that all eigenvalues of
$\displaystyle\nabla^2f_i(x;\mu)$ (cf. \eqref{hi}) are
$$\rho_1=\rho_2=\cdots=\rho_{n-1}=\dfrac{1}{g_i(x;\mu)},~\rho_n=\dfrac{\mu^2}{g_i(x;\mu)^{3}},$$
%
%
it follows from \eqref{gi} that \begin{equation*}\label{hbound}
\left\|\nabla^2f_i(x;\mu)\right\|\leq \frac{1}{\mu},~
i=1,2,\ldots,m.
\end{equation*}
The same argument as in \eqref{in2} shows
\begin{equation}\label{terma}\text{Term
A}\leq \dfrac{2m\epsilon}{\mu}, 
~\text{Term B}\leq {2m\epsilon}.\end{equation} 
Combining \eqref{1gra}, \eqref{in2}, and the fact $\left\|\nabla\tilde
    f(x;\mu)\right\|\leq1,$ we have
\begin{equation}\label{termc}
  \begin{array}{rcl}
    \text{Term C}&\leq& \left\|\nabla
    f(x;\mu)\left(\nabla f(x;\mu)-\nabla\tilde f(x;\mu)\right)^T+\left(\nabla f(x;\mu)-\nabla\tilde f(x;\mu)\right)\nabla\tilde
    f(x;\mu)^T\right\|\\[10pt]
    &\leq&\displaystyle \left\|\nabla
    f(x;\mu)\right\|\left\|\nabla f(x;\mu)-\nabla\tilde
    f(x;\mu)\right\|+\left\|\nabla f(x;\mu)-\nabla\tilde f(x;\mu)\right\|\left\|\nabla\tilde
    f(x;\mu)\right\|\\[10pt]
    &\leq&2\left\|\nabla f(x;\mu)-\nabla\tilde
    f(x;\mu)\right\|\leq{4m\epsilon}.
  \end{array}
\end{equation}Now we can use \eqref{term}, \eqref{terma}, and \eqref{termc} to conclude
    \begin{equation*}\label{in3}
\left\|\nabla^2f(x;\mu)-\nabla^2\tilde f(x;\mu)\right\|\leq
\dfrac{8m\epsilon}{\mu}=\frac{4\epsilon_3}{5}.
    \end{equation*} This completes the proof of Theorem \ref{errorth}.
\end{proof}

\subsection{Solving inexact Newton equation}
In the classical (line search) Newton-CG algorithm \cite{program,Yuan}, the
search direction is computed by applying the CG method to the Newton
equation
\begin{equation}\label{newton}
\nabla^2f(x;\mu)d=-\nabla f(x;\mu),
\end{equation}
until a direction $d$ is found to satisfy
\begin{equation*}\label{termination}
\left\|\nabla^2f(x;\mu)d+\nabla f(x;\mu)\right\|\leq
\eta(x;\mu)\left\|\nabla f(x;\mu)\right\|,
\end{equation*}
where $\eta(x;\mu)$ controls the solution accuracy. For instance, $\eta(x;\mu)$ can be chosen to be $$\min\left\{0.5,\,\sqrt{\left\|\nabla f(x;\mu)\right\|}\right\}.$$
A drawback of the classical Newton-CG algorithm when applied to solve the SEB problem with large $m$ and $n$ is that it is computationally expensive to obtain the Hessian and the Hessian-vector product.


 Fortunately, Theorem \ref{errorth} shows that $\nabla^2 \tilde f(x;\mu)$ and $\nabla \tilde f(x;\mu)$ are good approximations to $\nabla^2 f(x;\mu)$ and $\nabla f(x;\mu)$, respectively. Therefore, it is reasonable to replace the (exact) Newton equation \eqref{newton} with the inexact Newton equation
\begin{equation}\label{newtonapp}
\nabla^2\tilde f(x;\mu) \tilde d=-\nabla \tilde f(x;\mu).
\end{equation} Using the similar idea as in the classical Newton-CG algorithm, we do not solve \eqref{newtonapp} exactly but attempt to find a direction $\tilde d$ satisfying
\begin{equation}\label{termination}
\left\|\nabla^2\tilde f(x;\mu)\tilde d+\nabla \tilde f(x;\mu)\right\|\leq
\tilde\eta(x;\mu) \left\|\nabla \tilde f(x;\mu)\right\|,
\end{equation}
where $\tilde\eta(x;\mu)$ controls the solution accuracy. For instance, we can set $\tilde \eta(x;\mu)$ to be $$\tilde \eta(x;\mu)=\min\left\{0.5,\,\sqrt{\left\|\nabla \tilde f(x;\mu)\right\|}\right\}.$$

We apply the CG method to inexactly solve the linear equation \eqref{newtonapp} to obtain a search direction $\tilde d$ satisfying \eqref{termination}. The reasons for choosing the CG method for solving \eqref{newtonapp} are as follows. First, the matrix $\nabla^2\tilde f(x;\mu)$ is positive definite, which can be shown in the same way as in (iii) of Lemma \ref{yinli}, and the CG method is one of the most useful techniques for solving linear systems with positive definite coefficient matrices \cite{Yuan}. Second, in the inner CG iteration,
the Hessian-vector product $\nabla^2\tilde f(x;\mu)\tilde d$ is only required but not the Hessian $\nabla^2\tilde f(x;\mu)$ itself. This property makes the CG method particularly amenable to solve the linear equation \eqref{newtonapp}. Specifically, due to the special structure of
$\nabla^2\tilde f(x;\mu)$, the product $\nabla^2\tilde f(x;\mu)\tilde d$ can be obtained very fast for any given $\tilde d$. From
\eqref{ggi}, \eqref{hi}, \eqref{tildeg}, and \eqref{tildehessian},
simple calculations yield
\begin{equation}\label{hessian-vector}
\begin{array}{rcl}
  \nabla^2\tilde f(x;\mu)\tilde d&=&\displaystyle\sum_{i\in
{S(x;\mu,\epsilon)}}\left(\left(\frac{1}{\mu}-\frac{1}{{g_i(x;\mu)}}\right)\frac{\tilde\lambda_i(x;\mu)}{g_i(x;\mu)^2}(x-c_i)(x-c_i)^T\right)\tilde d\\
  &&+\displaystyle\sum_{i\in
{S(x;\mu,\epsilon)}}\frac{\tilde\lambda_i(x;\mu)}{{g_i(x;\mu)}}\tilde d-\frac{1}{\mu}\nabla
  \tilde f(x;\mu)\nabla \tilde f(x;\mu)^T\tilde d\\
  &=&\displaystyle\sum_{i\in
{S(x;\mu,\epsilon)}}\left[\left(\frac{1}{\mu}-\frac{1}{{g_i(x;\mu)}}\right)\frac{\tilde\lambda_i(x;\mu)}{g_i(x;\mu)^2}(x-c_i)\right](x-c_i)^T\tilde d\\
  &&+\displaystyle\left[\sum_{i\in
{S(x;\mu,\epsilon)}}\frac{\tilde \lambda_i(x;\mu)}{{g_i(x;\mu)}}\right]\tilde d-{\nabla \tilde f(x;\mu)^T\tilde d}\left[\frac{\nabla
  \tilde f(x;\mu)}{\mu}\right].
  \end{array}
\end{equation} The way of calculating $\nabla^2\tilde f(x;\mu)\tilde d$ by
\eqref{hessian-vector}\footnote{The terms in square brackets in \eqref{hessian-vector} are constants in the inner CG iteration, since they are not related to the variable $\tilde d$.} is typically different from the way of first
calculating $\nabla^2\tilde f(x;\mu)$ and then calculating
$\nabla^2\tilde f(x;\mu)\tilde d.$ The complexity of computing $\nabla^2\tilde f(x;\mu)\tilde d$ using the above two ways are $O(|S(x;\mu,\epsilon)|n)$ and $O\left((|S(x;\mu,\epsilon)|+n)n^2\right),$ respectively. It is worthwhile remarking that the computational complexity of calculating $\nabla^2 f(x;\mu)d$ using the above mentioned two ways of are $O(mn)$ and $O\left((m+n)n^2\right),$ respectively. Notice that $|S(x;\mu,\epsilon)|$ is usually much less than $m.$ Hence, computing $\nabla^2\tilde f(x;\mu)\tilde d$ by \eqref{hessian-vector} can sharply reduce the computational cost and simultaneously save a lot of memory (since we do not need to store the $n\times n$ matrix $\nabla^2\tilde f(x;\mu)$). 

Let $\tilde d_f$ be the obtained direction satisfying \eqref{termination} by applying the CG method to solve the linear equation \eqref{newtonapp} with the starting point $\tilde d_0=0.$ In the sequential, we state two properties of the direction $\tilde d_f$. These two properties shall be used late in global convergence analysis of the proposed algorithm.%

\begin{lemma}\label{descent}
Consider applying the CG method to solve \eqref{newtonapp} with the starting point $\tilde d_0=0.$ Suppose $\nabla \tilde f(x;\mu)\neq 0$ and $\tilde d_f$ is the obtained search direction satisfying \eqref{termination}. Then \begin{equation}\label{descentproperty}
 {\tilde d}_f^T\nabla \tilde f(x;\mu) = -\tilde d_f^T\nabla^2\tilde f(x;\mu)\tilde d_f<0.
\end{equation}
\end{lemma}
\begin{proof} Since the starting point $\tilde d_0=0,$ the final point $\tilde d_f$ in the CG iteration must have the form $\tilde d_f=\sum_{j=0}^{f-1}\tilde \alpha_j\tilde p_j$ \cite{program,Yuan}, where $\left\{\tilde \alpha_j\right\}_{j=0}^{f-1}$ and $\left\{\tilde p_j\right\}_{j=0}^{f-1}$ are step sizes and search directions in the CG iteration. Notice that $\nabla \tilde f(x;\mu)\neq 0,$ then $\tilde d_f\neq 0.$ Otherwise, substituting $\tilde d_f=0$ into \eqref{termination}, we shall get
$$\left\|\nabla \tilde f(x;\mu)\right\|\leq
\tilde\eta(x;\mu) \left\|\nabla \tilde f(x;\mu)\right\|\leq 0.5\left\|\nabla \tilde f(x;\mu)\right\|,$$ which contradicts the fact $\nabla \tilde f(x;\mu)\neq 0.$ Let \begin{equation}\label{tilder}\tilde r_f=\nabla^2\tilde f(x;\mu)\tilde d_f+\nabla \tilde f(x;\mu).\end{equation}
Then, it follows from \cite[Theorem 5.2]{program} that \begin{equation}\label{ortho}\tilde r_f^T\tilde p_j=0,~j=0,1,\ldots,f-1.\end{equation} Hence, \begin{align*}
  \tilde d_f^T\nabla \tilde f(x;\mu)=&~\tilde d_f^T\left(\tilde r_f-\nabla^2\tilde f(x;\mu)\tilde d_f\right)~(\text{from}~\eqref{tilder})\\
                             =&~\tilde d_f^T\tilde r_f-\tilde d_f^T\nabla^2\tilde f(x;\mu)\tilde d_f\\
                             =&~\sum_{j=0}^{f-1}\tilde \alpha_j\tilde p_j^T\tilde r_f-\tilde d_f^T\nabla^2\tilde f(x;\mu)\tilde d_f~(\text{substituting}~\tilde d_f=\sum_{j=0}^{f-1}\tilde \alpha_j\tilde p_j)\\
                             =&~-\tilde d_f^T\nabla^2\tilde f(x;\mu)\tilde d_f~(\text{from}~\eqref{ortho})\\
                             <&~0,
\end{align*}where the last inequality is due to positive definiteness of~$\nabla^2\tilde f(x;\mu)$ and the fact $\tilde d_f\neq0.$ The proof is completed. \end{proof}
%
%
\begin{lemma}
  Suppose $\tilde d$ satisfies \eqref{termination}, and $\tilde \sigma_{\max}(x;\mu)$ and $\tilde \sigma_{\min}(x;\mu)>0$ are the maximum and minimum eigenvalues of $\nabla^2\tilde f(x;\mu),$ respectively. Then
  \begin{equation}\label{length}
    \dfrac{1-\tilde\eta(x;\mu)}{\tilde \sigma_{\max}(x;\mu)}\left\|\nabla\tilde f(x;\mu)\right\|\leq \left\|\tilde d\right\|\leq \frac{1+\tilde\eta(x;\mu)}{\tilde \sigma_{\min}(x;\mu)}\left\|\nabla\tilde f(x;\mu)\right\|.
  \end{equation}
\end{lemma}

\begin{proof}
  Suppose the second inequality in \eqref{length} does not hold true, i.e., $$\left\|\tilde d\right\|> \frac{1+\tilde\eta(x;\mu)}{\tilde \sigma_{\min}(x;\mu)}\left\|\nabla\tilde f(x;\mu)\right\|.$$ Then 
  \begin{align*}
    \left\|\nabla^2\tilde f(x;\mu)\tilde d+\nabla \tilde f(x;\mu)\right\|\geq &~\left\|\nabla^2\tilde f(x;\mu)\tilde d\right\|-\left\|\nabla \tilde f(x;\mu)\right\|\\
                                                                  \geq &~\tilde \sigma_{\min}(x;\mu)\left\|\tilde d\right\|-\left\|\nabla \tilde f(x;\mu)\right\|\\
                                                                  > &~\left(1+\tilde\eta\left(x;\mu\right)\right)\left\|\nabla\tilde f(x;\mu)\right\|-\left\|\nabla \tilde f(x;\mu)\right\|\\
                                                                  =&~\tilde\eta(x;\mu)\left\|\nabla\tilde f(x;\mu)\right\|,
  \end{align*}which contradicts \eqref{termination}. Hence, the second inequality in \eqref{length} holds true.

  The similar argument shows the first inequality in \eqref{length} is also true. The proof is completed. 
\end{proof}


\subsection{Inexact Newton-CG algorithm}\label{secother}

When the smoothing parameter $\mu$ approaches zero,
$\exp\left(f_i(x;\mu)/\mu\right)$ tends to be very large.
The special care should be taken in computing $f(x;\mu)$ and $\tilde\lambda_i(x;\mu)$ to prevent overflow \cite{21}, i.e.,
\begin{align}
  f(x;\mu)&=f_{\infty}(x;\mu)+\mu\ln\left(\sum_{i=1}^m\exp(\left(f_i(x;\mu)-f_{\infty}(x;\mu)\right)/\mu)\right),\label{fcom}\\
  \tilde\lambda_i(x;\mu)&=\dfrac{\exp\left(\left(f_i(x;\mu)-f_{\infty}(x;\mu)\right)/\mu\right)}{\sum_{j\in
  {S(x;\mu,\epsilon)}}\exp(\left(f_j(x;\mu)-f_{\infty}(x;\mu)\right)/\mu)},~i\in {S(x;\mu,\epsilon)}.\label{lamcom}
\end{align}

%
%
%

Based on the above discussions, the specification of the proposed inexact Newton-CG algorithm for solving the SEB problem is given as follows.

\begin{algorithm}\caption{}\label{alg2}
\begin{algorithmic}[1]
\STATE Let $\epsilon_1,~c_1\in(0,1),~\beta\in(0,1),~\left\{\mu_k,~\epsilon_2(\mu_k),~\epsilon_3(\mu_k)\right\}_k,~x_{0,0}\in
\mathbb{R}^n$ be given and set $k=j=0.$
\REPEAT
\REPEAT
\STATE  Compute $S(x_{k,j};\mu_k,\frac{\mu_k\epsilon_3(\mu_k)}{10m})$ according to \eqref{set}.\label{line:epsilon3}
\STATE Compute the search direction $\tilde d_{k,j}$ by applying
the CG method to the inexact Newton equation $\nabla^2\tilde
f(x_{k,j};\mu_k)\tilde d=-\nabla\tilde f(x_{k,j};\mu_k)$ such that \begin{equation}\label{direction}\left\|\nabla^2\tilde f(x_{k,j};\mu_k)\tilde d_{k,j}+\nabla \tilde f(x_{k,j};\mu_k)\right\|\leq \tilde\eta_{k,j}\left\|\nabla \tilde f(x_{k,j};\mu_k)\right\|,\end{equation} where \begin{equation}\label{eta}
  \tilde \eta_{k,j}=\min\left\{0.5,\,\sqrt{\left\|\nabla \tilde f(x_{k,j};\mu_k)\right\|}\right\},
 \end{equation}
the Hessian-vector product $\nabla^2\tilde f(x_{k,j};\mu_k)\tilde d$ in the inner CG iteration, $\nabla\tilde
f(x_{k,j};\mu_k),$ and
$\tilde\lambda_i(x_{k,j};\mu_k)$ are computed by \eqref{hessian-vector}, \eqref{tildeg}, and \eqref{lamcom}, respectively.
\STATE Set $x_{k,j+1}=x_{k,j}+\alpha_{k,j}\tilde d_{k,j},$ where
$\alpha_{k,j}=\beta^l,$ with $\beta\in(0,1)$ and $l$ being the smallest integer satisfying the sufficient decrease condition
\begin{equation}\label{sufficient}
  f(x_{k,j}+\beta^l\tilde d_{k,j};\mu_k)\leq f(x_{k,j};\mu_k)+c_1\beta^l \tilde d_{k,j}^T\nabla \tilde f(x_{k,j};\mu_k),
\end{equation}where $f(x_{k,j};\mu_k)$ is computed by \eqref{fcom}.
\STATE Set $j=j+1.$
\UNTIL {$\left\|\nabla\tilde f(x_{k,j};\mu_k)\right\|\leq {\epsilon_2(\mu_k)}$}
\STATE Set $x_{k+1,0}=x_{k,j}$ and $k=k+1.$
\UNTIL {$\mu_k\leq\epsilon_1$}
  \end{algorithmic}
\end{algorithm}

The actual parameter values used for $\epsilon_1,~c_1,~\beta,~\left\{\mu_k,~\epsilon_2(\mu_k),~\epsilon_3(\mu_k)\right\}$ in Algorithm \ref{alg2} shall be given in Section \ref{experiment}. As we can see, all parameters in Algorithm \ref{alg2} are updated adaptively. For instance, the final iterate $x_{k,j}$ is set to be a warm starting point for the problem $\min_{x\in\mathbb{R}^n}f(x;\mu_{k+1}),$ and the tolerance $\epsilon_2(\mu_k)$ is set to be related to the approximation parameter $\mu_k.$


It is worthwhile pointing out that if we set $\epsilon_3(\mu_k)$ to be zero in line \ref{line:epsilon3} of the proposed Algorithm 2, then the proposed algorithm reduces to apply the classical Newton-CG algorithm to solve the smoothing approximation problem \eqref{sub}. Hence, the sequences generated by the proposed Algorithm \ref{alg2} converge to the unique solution of problem \eqref{nonsmooth} according to \cite[Theorem 3]{16}. In the next section, we shall show that even though the parameters $\epsilon_3(\mu_k)$ are positive, i.e., the gradient and Hessian-vector product are inexactly computed to reduce the computational cost, the proposed inexact Newton-CG Algorithm 2 is still globally convergent. 

%
%

\section{Convergence Analysis}\label{convergence}
In this section, we establish global convergence of the proposed Algorithm \ref{alg2} with an appropriate choice of parameters. For any $\mu>0,$ since $f(x;\mu)$ is strictly convex (see Lemma \ref{yinli}) and coercive in $x$, the level set \begin{equation}\label{omega}\Omega{(\mu)}=\left\{x\,|\,f(x;\mu)\leq f(x_{0,0};\mu_0)\right\}\end{equation} must be convex and bounded, where $x_{0,0}$ is the initial point in Algorithm \ref{alg2}. Furthermore, since the set $\left\{1,2,\ldots,m\right\}$ has a finite number of subsets, then
there must exist $\sigma_{\max}(\mu)\geq\sigma_{\min}(\mu)>0$ such that, for $\nabla^2\tilde f(x;\mu)$ defined on any proper subset of $\left\{1,2,\ldots,m\right\},$ we have \begin{equation}\label{sigma}\sigma_{\max}(\mu)I_n\succeq\nabla^2\tilde f(x;\mu)\succeq \sigma_{\min}(\mu) I_n,~\forall~x\in\Omega(\mu).\end{equation} As a particular case, we have$$\sigma_{\max}(\mu)I_n\succeq\nabla^2f(x;\mu)\succeq \sigma_{\min}(\mu) I_n,~\forall~x\in\Omega(\mu).$$ Before establishing global convergence of the proposed Algorithm \ref{alg2}, we first show that it is well defined. In particular, we prove that the proposed algorithm can always find a step length $\alpha_{k,j}=\beta^l$ satisfying \eqref{sufficient} in finite steps (see Lemma \ref{step}) and there exists $j_k$ such that $\left\|\nabla\tilde f(x_{k,j_k};\mu_k)\right\|\leq {\epsilon_2(\mu_k)}$ (see Lemma \ref{lemma-termination}).
\begin{lemma}\label{step}
  Suppose $\left\|\nabla\tilde f(x_{k,j};\mu_k)\right\|> {\epsilon_2(\mu_k)},$ and set \begin{equation}\label{epsilon3}\epsilon_3(\mu_k)\leq \frac{\epsilon_2(\mu_k)}{c_2(\mu_k)},\end{equation} where \begin{equation}\label{c2}c_2(\mu_k)>\frac{6\mu_k\sigma_{\max}^2(\mu_k)}{5(1-c_1)\sigma_{\min}^2(\mu_k)}.\end{equation} Then the step length $\alpha_{k,j}$ satisfying the sufficient decrease condition \eqref{sufficient} can be found in
  $\left\lceil\frac{\ln\left(\bar\alpha(\mu_k)\right)}{\ln\left(\beta\right)}\right\rceil$ steps, and \begin{equation}\label{lowerbound}
  \alpha_{k,j}\geq \beta\bar\alpha(\mu_k),~\forall~j,
\end{equation}where \begin{equation}\label{alphaupper}\bar\alpha(\mu_k)=\frac{2(1-c_1)\sigma_{\min}^3(\mu_k)}{9\sigma_{\max}^3(\mu_k)}-\frac{4\mu_k\sigma_{\min}(\mu_k)}{15\sigma_{\max}(\mu_k)c_2(\mu_k)}>0.\end{equation} \end{lemma}
\begin{proof} By the mean value theorem, there exists $s\in(0,1)$ such that
\begin{equation}\label{ABC}
\begin{array}{rl}
  &f(x_{k,j}+\alpha \tilde d_{k,j};\mu_k)-f(x_{k,j};\mu_k)-c_1\alpha \tilde d_{k,j}^T\nabla\tilde f(x_{k,j};\mu_k)\\[5pt]
 =&\alpha \underbrace{\tilde d_{k,j}^T\left(\nabla f(x_{k,j};\mu_k)-\nabla\tilde f\left(x_{k,j};\mu_k\right)\right)}_{\text{Term A}}+\left(1-c_1\right)\alpha \underbrace{\tilde d_{k,j}^T\nabla\tilde f(x_{k,j};\mu_k)}_{\text{Term B}}\\
 &+\frac{\alpha^2}{2}\underbrace{\tilde d_{k,j}^T\nabla^2 f\left(x_{k,j}+\alpha s \tilde d_{k,j};\mu_k\right)\tilde d_{k,j}}_{\text{Term C}}.
 \end{array}
\end{equation}
Next, we upper bound Term A, Term B, and Term C in the above, respectively.
It follows from \eqref{gerror} and \eqref{epsilon3} that
\begin{equation}\label{mepsilon}\left\|\nabla f(x;\mu_k)-\nabla \tilde
f(x;\mu_k)\right\|\leq \frac{\mu_k\epsilon_3(\mu_k)}{5}\leq \frac{\mu_k\epsilon_2(\mu_k)}{5c_2(\mu_k)}.\end{equation} Furthermore, since
\begin{equation}\label{rate1}\left\|\nabla\tilde f(x_{k,j};\mu_k)\right\|> {\epsilon_2(\mu_k)},\end{equation}there holds
$$\left\|\nabla f(x_{k,j};\mu_k)-\nabla\tilde f(x_{k,j};\mu_k)\right\|\leq \frac{\mu_k}{5c_2(\mu_k)}\left\|\nabla\tilde f(x_{k,j};\mu_k)\right\|.$$
Combining the above inequality, the second inequality of \eqref{length}, and \eqref{eta}, we have
\begin{align}
  \tilde d_{k,j}^T\left(\nabla f(x_{k,j};\mu_k)-\nabla\tilde f(x_{k,j};\mu_k)\right) &\leq \displaystyle\left\|\tilde d_{k,j}\right\|\left\|\nabla f(x_{k,j};\mu_k)-\nabla\tilde f(x_{k,j};\mu_k)\right\|\nonumber\\
  &\leq {\frac{3\mu_k}{10c_2(\mu_k)\sigma_{\min}(\mu_k)} \left\|\nabla\tilde f(x_{k,j};\mu_k)\right\|^2}.\label{bb1}
\end{align}
From the first inequality of \eqref{length}, \eqref{descentproperty}, \eqref{eta} and \eqref{sigma}, we obtain
\begin{align}\label{bb2}
  \tilde d_{k,j}^T\nabla\tilde f(x_{k,j};\mu_k)~=&~-\tilde d_{k,j}^T\nabla^2\tilde f(x_{k,j};\mu_k)\tilde d_{k,j}~(\text{from}~\eqref{descentproperty})\\\nonumber
                                        \leq &~-\sigma_{\min}(\mu_k)\left\|\tilde d_{k,j}\right\|^2~(\text{from}~\eqref{sigma})\\\nonumber
                                        \leq &~\frac{-\sigma_{\min}(\mu_k)}{4\sigma_{\max}^2(\mu_k)}\left\|\nabla\tilde f(x_{k,j};\mu_k)\right\|^2.~(\text{from}~\eqref{length}~\text{and}~\eqref{eta})\nonumber
\end{align}The similar argument as in \eqref{bb2} shows that Term C in \eqref{ABC} can be upper bounded by
\begin{equation}\label{bb3}
  \tilde d_{k,j}^T\nabla^2 f(x_{k,j}+\alpha s \tilde d_{k,j};\mu_k)\tilde d_{k,j}\leq\frac{9\sigma_{\max}(\mu_k)}{4\sigma_{\min}^2(\mu_k)} \left\|\nabla\tilde f(x_{k,j};\mu_k)\right\|^2.
\end{equation}
By combining \eqref{ABC},~\eqref{bb1},~\eqref{bb2}, and \eqref{bb3}, we obtain
\begin{align*}\label{errABC}
&~f(x_{k,j}+\alpha \tilde d_{k,j};\mu_k)-f(x_{k,j};\mu_k)-c_1\alpha \tilde d_{k,j}^T\nabla\tilde f(x_{k,j};\mu_k)\\\nonumber
\leq&~\alpha\left\|\nabla\tilde f(x_{k,j};\mu_k)\right\|^2\left(\frac{3\mu_k}{10c_2(\mu_k)\sigma_{\min}(\mu_k)}- \frac{(1-c_1)\sigma_{\min}(\mu_k)}{4\sigma_{\max}^2(\mu_k)}+{\alpha}{\frac{9\sigma_{\max}(\mu_k)}{8\sigma_{\min}^2(\mu_k)}}\right).\nonumber\end{align*}
Consequently, it follows from \eqref{alphaupper} that, for any $l$ such that $\beta^l\leq \bar\alpha(\mu_k),$ $\beta^l$ satisfies the inequality \eqref{sufficient}, and the inequality \eqref{lowerbound} holds true.
This completes the proof of Lemma \ref{step}.
%
%
%
\end{proof}

\begin{lemma}\label{lemma-termination}
  Suppose $\mu=\mu_k$ and $\left\{x_{k,j}\right\}$ be the sequence generated by Algorithm \ref{alg2}. Then there must exist $j_k$ such that \begin{equation}\label{epsilon2}\left\|\nabla\tilde f(x_{k,j_k};\mu_k)\right\|\leq {\epsilon_2(\mu_k)}.\end{equation}
\end{lemma}
\begin{proof}
  We prove Lemma \ref{lemma-termination} by contradiction, i.e., suppose \begin{equation}\label{contradiction}\left\|\nabla\tilde f(x_{k,j};\mu_k)\right\|> {\epsilon_2(\mu_k)},~\forall~j=1,2,\ldots.\end{equation}
  Since $f(x,\mu_k)$ is lower bounded (by zero), it follows that $$+\infty>\sum_{j=0}^{+\infty}\left(f(x_{k,j};\mu_k)-f(x_{k,j+1};\mu_k)\right)=\sum_{j=0}^{+\infty}\left(f(x_{k,j};\mu_k)-f(x_{k,j}+\alpha_{k,j}\tilde d_{k,j};\mu_k)\right).$$ Moreover, since
\begin{align*}
+\infty>&~\sum_{j=0}^{+\infty}\left(f(x_{k,j};\mu_k)-f(x_{k,j}+\alpha_{k,j}\tilde d_{k,j};\mu_k)\right)\\
\geq&~\sum_{j=0}^{+\infty}\left(-c_1\alpha_{k,j} \tilde d_{k,j}^T\nabla \tilde f(x_{k,j};\mu_k)\right)~(\text{from}~\eqref{sufficient})\\
                                                     \geq &~\sum_{j=0}^{+\infty}\left(c_1 \beta \bar\alpha(\mu_k) \tilde d_{k,j}^T\nabla^2 \tilde f(x_{k,j};\mu_k)\tilde d_{k,j}\right)~(\text{from}~\eqref{descentproperty}~\text{and}~\eqref{lowerbound})\\
                                                     \geq &~{c_1 \beta \bar\alpha(\mu_k)\sigma_{\min}(\mu_k)}\sum_{j=0}^{+\infty}\|\tilde d_{k,j}\|^2,
\end{align*}it follows that
$\displaystyle \lim_{j\rightarrow+\infty}\|\tilde d_{k,j}\|=0.$ Taking limits from both sides of \eqref{direction}, we obtain 
\begin{equation*}\label{limitgradient}\lim_{j\rightarrow+\infty}\left\|\nabla\tilde f(x_{k,j};\mu_k)\right\|= 0,\end{equation*} which contradicts \eqref{contradiction}. Hence, Lemma \ref{lemma-termination} is true.
\end{proof}
Now, we are ready to present the global convergence result of Algorithm \ref{alg2}.

\begin{theorem}\label{thm-convergence} Let $\epsilon_1=0,$ $\displaystyle \lim_{k\rightarrow+\infty}\epsilon_2(\mu_k)= 0,$ and $\epsilon_3(\mu_k)$ satisfies \eqref{epsilon3} for all $k$ in Algorithm \ref{alg2}. Suppose that $\left\{x_{k,j_k}\right\}$ be the sequences generated by Algorithm \ref{alg2} satisfying \eqref{epsilon2} and $x^*$ be the unique solution to problem \eqref{nonsmooth}. Then $$\lim_{k\rightarrow \infty} x_{k,j_k}=x^*.$$
\end{theorem}
\begin{proof}
Recalling the definition of $\Omega(\mu)$ (cf. \eqref{omega}), it follows from part 2 of Lemma \ref{yinli} that
 $$\Omega(\mu_k)\subset\Omega:=\left\{x\,|\,f(x)\leq f(x_{0,0};\mu_0)\right\},~\forall~k\geq 0.$$ Since $f(x)$ is coercive, we know that $\Omega$ is bounded.
~From part 1 of Lemma \ref{yinli} and \eqref{sufficient}, we have
  $$f(x_{k,j_k};\mu_k)=f(x_{k+1,0};\mu_{k})>f(x_{k+1,0};\mu_{k+1})\geq f(x_{k+1,1};\mu_{k+1})\geq\cdots\geq f(x_{k+1,j_{k+1}};\mu_{k+1}).$$
 Hence, the function values $\{f(x_{k,j_k};\mu_k)\}$ are decreasing, and the sequences $\{x_{k,j_k}\}$ lie in the bounded set $\Omega$.
 Then there must exist an accumulation point for $\{x_{k,j_k}\}$. Let $\bar{x}$ denote an accumulation point such that
$$\bar{x}=\lim_{k \in {\cal {K}}, k\rightarrow\infty}x_{k,j_k}$$
  for some subsequence indexed by $\cal {K}$. Since $\{f(x_{k,j_k};\mu_k)\}$ are decreasing and bounded below (by zero), it follows that $\displaystyle \lim_{k\rightarrow +\infty}f(x_{k,j_k};\mu_k)=f(\bar{x}).$
%

Next, we show that $\|\nabla f(x_{k,j_k};\mu_k)\|\rightarrow 0.$ In fact, it follows from \eqref{c2}, \eqref{mepsilon} and \eqref{epsilon2} that we have
\begin{align}\label{errorbound}\|\nabla f(x_{k,j_k};\mu_k)\|\leq \|\nabla f(x_{k,j_k};\mu_k)-\nabla \tilde f(x_{k,j_k};\mu_k)\|+\|\nabla \tilde f(x_{k,j_k};\mu_k)\| \leq  2\epsilon_2(\mu_k).
                                       \end{align}
Letting $k$ go to infinity, we obtain the desired result $\|\nabla f(x_{k,j_k};\mu_k)\|\rightarrow 0.$ According to \cite[Lemma 2, Theorem 3]{16}, we know $\bar x=x^*.$ This completes the proof of Theorem \ref{thm-convergence}.
\end{proof}

\section{{Numerical Results}}\label{experiment}

In this section, the proposed inexact Newton-CG algorithm (Algorithm \ref{alg2}) was implemented and the
numerical experiments were done on a personal computer with Intel Core i7-4790K CPU (4.00 GHz) and 16GB of memory. We implemented our codes in C language and compared it with the state-of-the-art Algorithm 1 \cite{16} and the classical Newton-CG algorithm.
The test problems are generated randomly. Similar to \cite{16}, we use the following pseudo-random sequences:
$$\psi_0=7,~\psi_{i+1}=\left(445\psi_i+1\right)\!\!\!\!\!\mod4096,~\bar\psi_i=\psi_i/40.96,~i=1,2,\ldots$$
The elements of $r_i$ and $c_i,~i=1,2,\ldots,m,$ are successively set to
$\bar\psi_1,\bar\psi_2,\ldots,$ in the order:
$$r_1,c_1(1),c_1(2),\ldots,c_1(n),r_2,c_2(1),c_2(2),\ldots,c_2(n),\ldots,r_m,c_m(1),c_m(2),\ldots,c_m(n).$$
 Different
scales of the SEB problem are tested and 
the parameters used in Algorithm \ref{alg2} are set to be
$$\epsilon_1=1\text{E}-6,~c_1=1\text{E}-4,~\beta=0.5,~x_{0,0}=0,$$$$\mu_k=\left(0.1\right)^k,~\epsilon_2(\mu_k)=\max\left\{1\text{E}-5,\min\left\{1\text{E}-1;\mu_k/10\right\}\right\},~\epsilon_3(\mu_k)=1\text{E}-2,~k=0,1,\ldots,6.$$

 The simulation results are summarized in Table \ref{table1}, Table \ref{table2}, and Table \ref{table3},
~where $n$ denotes the dimension of the Euclidean space, $m$ denotes the number
of balls, \textbf{Obj Value} denotes the value of the objective
function in \eqref{nonsmooth} at the final iterate, and \textbf{Time} denotes the CPU
time in seconds for solving the corresponding SEB
problem.

\begin{table*}
\caption{{Performance comparison of proposed Algorithm \ref{alg2}, classical Newton-CG algorithm, and Algorithm \ref{alg1} in \cite{16} with different large $m$ and $n=1000/2000.$}} \label{table1} \centering
\liuhao{\begin{tabular}{ccccccc}
\hline {Problem}&  \multicolumn{2}{c}{Proposed Algorithm \ref{alg2}} & \multicolumn{2}{c}{Classical Newton-CG Algorithm}& \multicolumn{2}{c}{Algorithm \ref{alg1} \cite{16}}\\
\cline{1-1}\cline{2-3}\cline{4-5}\cline{6-7}
 $(m,n)$&Time&Obj Value&Time&Obj value&Time&Obj value\\\hline
%

(10000,1000)&  6.01552E+00   &1.0228463348E+03 &   6.00472E+01 & 1.0228463348E+03  &   5.90720E+01 & 1.0228463348E+03 \\
(20000,1000)&  1.10170E+01  &1.0228463347E+03 &  1.16408E+02  &  1.0228463347E+03 &   1.08360E+02 & 1.0228463347E+03 \\

(30000,1000)&  1.92800E+01  &1.0228463347E+03 &   1.70676E+02  &  1.0228463347e+03 &   1.59982E+02 & 1.0228463347E+03 \\
(40000,1000)&  2.35979E+01  &1.0228463346E+03 &   2.49327E+02 &   1.0228463346e+03 &   2.14136E+02 & 1.0228463346E+03\\

(50000,1000)&  3.02345E+01  &1.0228463347E+03 &    2.92019E+02 &   1.0228463347e+03 &   2.88675E+02 & 1.0228463347E+03\\
(100000,1000)&   5.61113E+01  &1.0228463347E+03 &   5.58410E+02  &  1.0228463347e+03  &   5.86357E+02 & 1.0228463347E+03\\

(10000,2000)&  1.52080E+01   &1.3984577651E+03 &   1.63971E+02 & 1.3984577651E+03  &   1.46725E+02 & 1.3984577651E+03 \\

(20000,2000)&  3.04284E+01  &1.3984577651E+03 &  3.35675E+02  &  1.3984577651E+03 &   2.43072E+02 & 1.3984577651E+03 \\

(30000,2000)&  4.53796E+01  &1.3984577649E+03 &   4.41727E+02  &  1.3984577649E+03 &   4.17976E+02 & 1.3984577649E+03 \\
(40000,2000)&  5.89272E+01  &1.3984577650E+03 &   5.84628E+02 &   1.3984577650E+03 &   5.57352E+02 & 1.3984577650E+03\\

(50000,2000)&  7.38869E+01  &1.3984577650E+03 &    7.32414E+02 &   1.3984577650E+03 &   7.18307E+02 & 1.3984577650E+03\\
(100000,2000)&   1.45665E+02  &1.3984577650E+03 &   1.68460E+03  &  1.3984577650E+03  &   1.32608E+03 & 1.3984577650E+03\\

    \hline
\end{tabular}
}
\end{table*}

%
%
%
%
%
%

\begin{table*}
\caption{{Performance comparison of proposed Algorithm \ref{alg2}, Algorithm \ref{alg1} in \cite{16}, and classical Newton-CG algorithm with different large/huge $m$ and $n=100.$}} \label{table2} \centering
\liuhao{\begin{tabular}{ccccccc}
\hline {Problem}&  \multicolumn{2}{c}{Proposed Algorithm \ref{alg2}} & \multicolumn{2}{c}{Classical Newton-CG Algorithm} & \multicolumn{2}{c}{Algorithm \ref{alg1} \cite{16}}\\
\cline{1-1}\cline{2-3}\cline{4-5}\cline{6-7}
 $(m,n)$&Time&Obj Value&Time&Obj value&Time&Obj value\\\hline

(16000,100)&  5.92168E-01   &4.0409180661E+02 &   5.07334E+00 & 4.0409180661E+02  &   7.54851E+00 & 4.0409180661E+02 \\
(32000,100)&  1.27262E+00  &4.0409180660E+02 &  1.09639E+01  &  4.0409180660E+02 &   1.30855E+01 & 4.0409180660E+02 \\

(64000,100)&  2.49998E+00  &4.0409180660E+02 &   2.03190E+01  &  4.0409180660E+02 &   2.87963E+01 & 4.0409180660E+02 \\
(128000,100)&  4.75222E+00  &4.0409180660E+02 &   4.17729E+01 &   4.0409180660E+02 &   5.39767E+01 & 4.0409180660E+02\\

(256000,100)&  1.05534E+01  &4.0409180662E+02 &    9.10787E+01 &   4.0409180662E+02 &   1.12828E+02 & 4.0409180662E+02\\
(512000,100)&   2.14911E+01  &4.0409180662E+02 &   1.76339E+02  &  4.0409180662E+02  &   2.10324E+02 & 4.0409180662E+02\\

(1024000,100)&  4.51870E+01  &4.0409180662E+02 &    3.38687E+02 &   4.0409180662E+02 &   4.28966E+02 & 4.0409180662E+02\\
(2048000,100)&   9.00532E+01  &4.0409180662E+02 &   6.74268E+02  &  4.0409180662E+02  &   9.56397E+02 & 4.0409180662E+02\\

    \hline
\end{tabular}
}
\end{table*}

\begin{table*}
\caption{{Performance comparison of proposed Algorithm \ref{alg2}, Algorithm \ref{alg1} in \cite{16}, and classical Newton-CG algorithm with different large $m$ and different large $n.$}} \label{table3} \centering
\liuhao{\begin{tabular}{ccccccc}
\hline {Problem}&  \multicolumn{2}{c}{Proposed Algorithm \ref{alg2}} & \multicolumn{2}{c}{Classical Newton-CG Algorithm} & \multicolumn{2}{c}{Algorithm \ref{alg1} \cite{16}}\\
\cline{1-1}\cline{2-3}\cline{4-5}\cline{6-7}
 $(m,n)$&Time&Obj Value&Time&Obj value&Time&Obj value\\\hline

(2000,5000)&  1.62869E+01   &2.1340381607E+03 &   1.04678E+02 & 2.1340381607E+03  &   6.79232E+01 & 2.1340381608E+03 \\

(2000,10000)&  4.07598E+01  &2.9778347203E+03 &  2.43027E+02  &  2.9778347203E+03 &   1.41727E+02 & 2.9778347203E+03 \\

(5000,5000)&  3.21717E+01  &2.1377978300E+03 &   3.34972E+02  &  2.1377978300E+03 &   1.60278E+02 & 2.1377978300E+03 \\
(5000,10000)&  1.18895E+02  &2.9814491291E+03 &   9.20224E+02 &   2.9814491291E+03 &   3.69496E+02 & 2.9814491291E+03\\

(8000,7000)&  1.56516E+02  &2.5108384309E+03 &    1.50597E+03 &   2.5108384309E+03 &   4.17278E+02 & 2.5108384309E+03\\

(100000,8000)&   1.54017E+02  &2.6749515342E+03 &   1.24485E+03  &  2.6749515342E+03  &   5.43961E+02 & 2.6749515342E+03\\

(100000,10000)&   2.16453E+02  &2.9814491291E+03 &   1.66959E+03  &  2.9814491291E+03  &   8.13515E+02 & 2.9814491291E+03\\

(200000,10000)&   4.13139E+02  &2.9814491290E+03 &   3.04470E+03  &  2.9814491290E+03  &   1.59488E+03 & 2.9814491290E+03\\

\hline
\end{tabular}
}
\end{table*}

It can be seen from the three tables that  
the proposed Algorithm \ref{alg2} significantly outperforms Algorithm \ref{alg1} in \cite{16} and the classical Newton-CG algorithm in terms of the CPU time to find the same solution. In particular, Algorithm \ref{alg1} and the classical Newton-CG algorithm take $8$ and $10$ times more CPU time than proposed Algorithm \ref{alg2} in average to find the same solution, respectively. The proposed algorithm is able to solve the SEB problem with $m=2048000$ and $n=100$ within about $90$ seconds, while the classical Newton-CG algorithm and Algorithm 1 in \cite{16} need $674$ and $956$ seconds to do so, respectively. The proposed inexact Newton-CG algorithm significantly improves the classical Newton-CG algorithm by computing the gradient and Hessian-vector product in an inexact fashion, which dramatically reduces the CPU time compared to the exact computations.

\begin{figure}
     \includegraphics[width=0.85\textwidth]{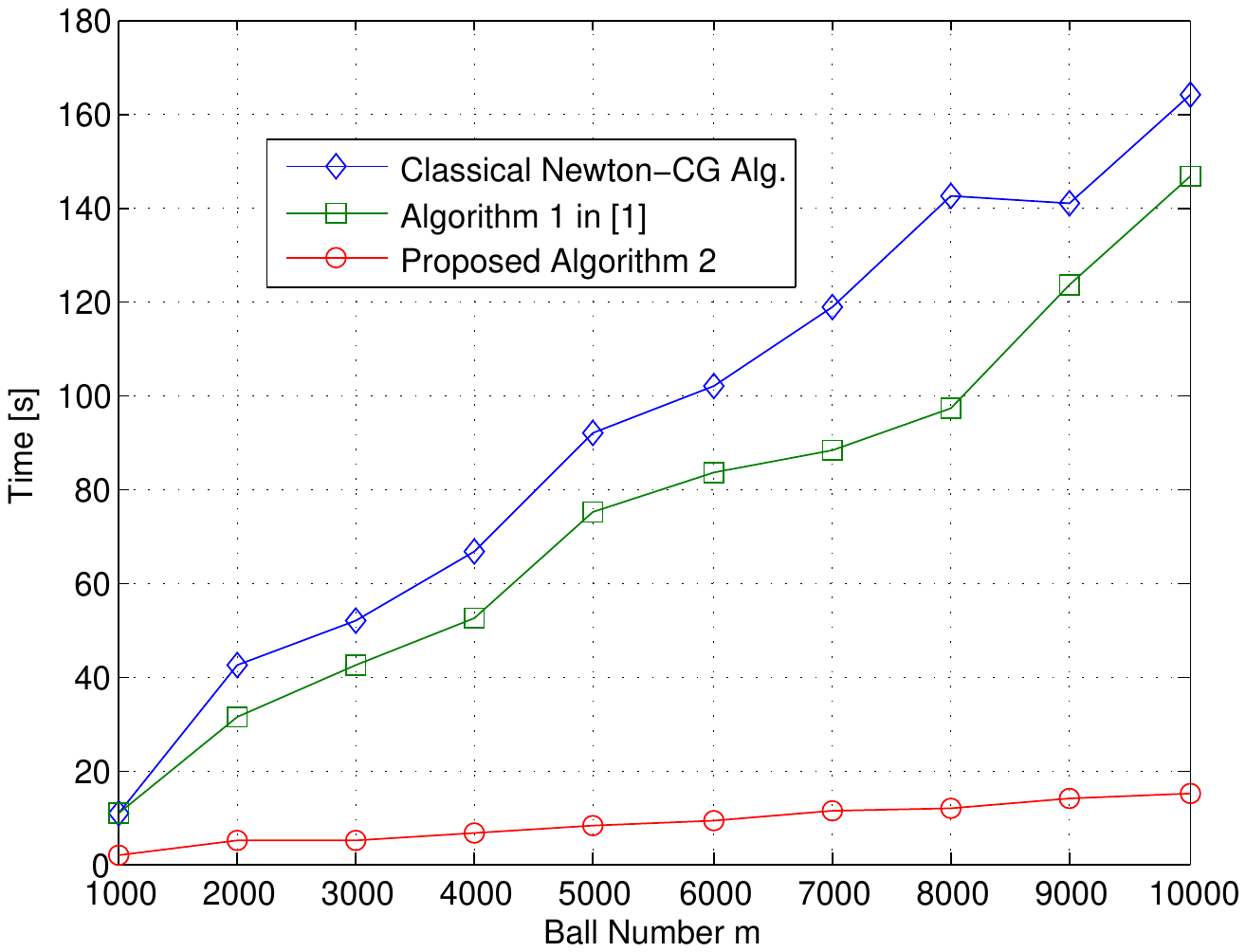}
     \caption{Time comparison of proposed Algorithm \ref{alg2}, Algorithm \ref{alg1} in \cite{16}, and classical Newton-CG algorithm with different large $m$ and fixed $n=2000.$}
     \label{Time}
     \end{figure}

We also plot the CPU time comparison of proposed Algorithm \ref{alg2}, Algorithm \ref{alg1} in \cite{16}, and classical Newton-CG algorithm with different large $m$ and fixed $n=2000$ as Fig. \ref{Time}. It can be observed from Fig. \ref{Time} that for fixed $n=2000$, the CPU time of all of three algorithms grow (approximately) linearly with $m.$ However, the CPU time of both Algorithm \ref{alg1} and the classical Newton-CG algorithm grows much faster than that of proposed Algorithm \ref{alg2}.

In a nutshell, our numerical simulation results show that the proposed inexact Newton-CG algorithm is particularly amenable to solve the SEB problem of large dimensions. First, the gradient and Hessian-vector product are inexactly computed at each iteration of the proposed algorithm by exploiting the (approximate) sparsity structure of the log-exponential aggregation function. This dramatically reduces the computational cost compared to compute the gradient and Hessian-vector product exactly and thus makes the proposed algorithm well suited to solve the SEB problem with large $m.$
Second, 
at each iteration, the proposed algorithm computes the search direction by applying the CG method to solve the inexact Newton equation in an inexact fashion. 
This makes the proposed algorithm also very attractive to solving the SEB problem with large $n.$ 

{\section{Conclusions}\label{conclusion}
In this paper, we developed a computationally efficient inexact Newton-CG algorithm for the SEB problem of large dimensions, which finds wide applications in pattern recognition, machine learning, support vector machines and so on. The key difference between the proposed inexact Newton-CG algorithm and the classical Newton-CG algorithm is that the gradient and the Hessian-vector product are inexactly computed in the proposed algorithm by exploiting the special (approximate) sparsity structure of its log-exponential aggregation function. We proposed an adaptive criterion of inexactly computing the gradient/Hessian and also established global convergence of the proposed algorithm. Simulation results show that the proposed algorithm significantly outperforms the classical Newton-CG algorithm and the state-of-the-art algorithm in \cite{16} in terms of the computational CPU time. Although we focused on the SEB problem in this paper, the proposed algorithm can be applied to solve other min-max problems in \cite{minimax1,minimax2,minimax3,coordinated,simo,ICC}.

\section{Acknowledgments}
The authors wish to thank Professor Ya-xiang Yuan and Professor Yu-Hong Dai of State Key Laboratory of
Scientific and Engineering Computing, Academy of Mathematics and Systems Science, Chinese Academy of Sciences,
for their helpful comments on the paper. The authors also thank Professor Guanglu Zhou of Department of Mathematics and Statistics, Curtin University, for sharing the code of Algorithm \ref{alg1} in \cite{16}.

\end{document}